\documentclass[runningheads,orivec]{llncs}
\usepackage[T1]{fontenc}
\usepackage{amssymb,amsmath}
\usepackage{graphicx} 
\usepackage{tikz}
\usepackage{xcolor}
\usepackage{float}
\usetikzlibrary{arrows.meta}
\usepackage{comment}
\spnewtheorem{observation}[theorem]{Observation}{}{}

\definecolor{orcidgreen}{RGB}{166,206,57}

\begin{document}

\title{Eccentric Connectivity Index of Cartesian product and Strong product of Strongly Connected Digraphs\thanks{First author is also a part-time research scholar in the Department of Mathematics, Government College Chittur, Palakkad, Kerala, India - 678104, affiliated to University of Calicut, Thenhipalam, Kerala, India - 673635.}}
\titlerunning{Eccentric Connectivity Index $\ldots$}
\author{Vysakh Chakooth\inst{1}\orcidID{0009-0007-9176-3845} \and{Prasanth G. Narasimha-Shenoi\inst{2,3}\orcidID{0000-0002-5850-5410}}
}
\authorrunning{Vysakh Chakooth et al.}
\institute{Department of Mathematics, NSS College Ottappalam, Palakkad, Kerala, India - 679103
\email{v2sakh@gmail.com}
\and
Department of Mathematics, Government College Chittur, Palakkad, Kerala, India - 678104,  \email{prasanthgns@gmail.com}\and 
Department of Collegiate Education, Government of Kerala,
Thiruvananthapuram, Kerala India - 695033
}

\maketitle

\begin{abstract}
Let $G=(V,E)$ be a graph. The \emph{eccentric connectivity index} of $G$ is defined as $$\xi^C(G)=\sum_{u\in V(G)}d_uecc(u)$$ where $d_u$ and $ecc(u)$ are the degree and eccentricity of $u$, respectively. For a strongly connected digraph $D=(V,A)$, the eccentric connectivity index is defined as $$\xi^C(D)=\frac{1}{2}\sum_{u\in V(D)}(d_u^++d_u^-)mecc(u)$$ where $d_u^+$ and $d_u^-$ are the out-degree and in-degree of $u$, respectively, and $mecc(u)$ is its m-eccentricity with respect to the maximum distance $md(u,v)=\max\{\vec d(u,v),\vec d(v,u)\}$.

In this article, give formulas and bounds for the eccentric connectivity index of Cartesian and strong products of strongly connected digraphs and discuss the corresponding equality cases. Also, an attempt is made to study the self-centeredness of these products and establish conditions under which the Cartesian and strong products are self-centered.The results are extended to products of several digraphs.

\keywords{Eccentric connectivity index \and maximum distance \and Cartesian product of digraphs \and strong product of digraphs.}

\end{abstract}

\section{Introduction}
 The \emph{eccentric connectivity index} of simple graphs introduced by Sharma et al., in \cite{sharma1997eccentric}, is a distance–cum–degree based descriptor defined as $$\displaystyle{\xi^C(G) = \sum_{u \in V}  d_u \;ecc(u)}$$ where $d_u$ is the degree and $ecc (u)$ is the eccentricity of the vertex $u$ of the graph $G=(V,E)$. The eccentric connectivity index has been extensively studied in undirected graphs by various authors; see, for example~ \cite{alizadeh2020relation,ashrafi2011eccentric,dovslic2014eccentric,hauweele2019maximum,morgan2011eccentric,zhang2014maximal,zhang2012minimal,zhou2010eccentric} and in composite graphs \cite{azari2022study} due to its strong correlation with molecular properties and  has been widely used as a molecular descriptor in QSAR and  QSPR studies \cite{dutt2010improved,idrees2019eccentricity,mansha2025eccentric} to correlate molecular topology with physicochemical and biological properties.
However, many real-world systems, such as communication, transportation, and biological networks, are inherently directional, leading to asymmetric distances. For this reason, several distance–based indices \cite{dankelmann2023wiener,knor2016digraphs} and degree–based indices \cite{ganie2025bounds,monsalve2021vertex} have already been extended to digraphs , very recently the eccentric connectivity index with respect to metric maximum distance is extended to digraphs in \cite{chakooth2025eccentric}. 

The Cartesian and strong products offer fundamental ways to construct larger digraphs from smaller ones, making it natural to examine how the eccentric connectivity index behaves under these operations. Studying these products also allows the index of the resulting digraph to be related to structural and distance parameters of the factor digraphs. This provides a natural way to determine the eccentric connectivity index of larger digraphs from that of their factors and may be useful in the study of more complex directed networks.

In this paper, the eccentric connectivity index, with respect to the metric maximum distance defined in \cite{chartrand1997distance}, is studied for Cartesian and strong products of strongly connected digraphs. Bounds and equality cases are obtained for both products, and the results are extended to products of several digraphs and their powers. A comparison between the eccentric connectivity indices of the Cartesian and strong products is also presented.


\section{Preliminaries}\label{sec:prelim}
A directed graph (or digraph) $D = (V, A)$ consists of a non-empty finite set $V$ of vertices, and a finite set $A$ of ordered pairs of distinct vertices, called \textit{arcs}, see \cite{bang2008digraphs}. An arc from vertex $u$ to vertex $v$ is denoted by $(u,v)$, indicating a directed edge from $u$ to $v$, and $u$ is called the tail and $v$ is called the head. The \emph{out-degree} of a vertex $u$, denoted by $d^+_u$, is the number of arcs originating from $u$, that is, arcs of the form $(u,v)$. Similarly, the \emph{in-degree} of $u$, denoted by $d^-_u$, is the number of arcs directed towards $u$, That is, arcs of the form $(v,u)$. A digraph $D$ is $r-$regular if and only if $d^+_u=r=d^-_u,$ for all $u \in V(D)$.
 
A \textit{directed walk} (or diwalk) in a digraph $D$ is an alternating sequence $ W = x_1 a_1 x_2 a_2 \ldots x_{k-1} a_{k-1} x_k,
$ where $x_i$ are vertices and $a_j$ are arcs such that each arc $ a_i$ has tail $ x_i $ and head $ x_{i+1}$ for all $ i \in 1,\ldots, k-1$. We denote this by $x_1 \rightarrow{x_2}\rightarrow \cdots \rightarrow{x_{k-1} \rightarrow{x_k}}$. If all the vertices in the walk are distinct, $W$  is called a \textit{directed path} (or dipath). If $x_1 = x_k $, the vertices $ x_1, \ldots, x_{k-1} $ are distinct, and $ k\geq 3 $, then $ W $ is referred to as a \textit{directed cycle} (or dicycle). For a detailed study see \cite{bang2018classes}.  

A digraph $D$ is said to be \textit{strongly connected} (or strong) if, for every pair of distinct vertices $(u, v)$, there exists a directed path from $u$ to $v$. The length of a directed path is the number of arcs it contains.  For any pair of vertices $(u, v)$ in a strongly connected digraph $D$, the shortest directed $u, v$- path is called a directed geodesic, and its length, that is, the number of arcs in it is termed the directed distance, denoted by $\overrightarrow{d}(u, v)$.


 
 If a digraph is not strongly connected, the directed distance between two vertices need not be finite. Therefore, throughout this paper we consider strongly connected digraphs and use the metric ‘\textit{maximum distance}’, abbreviated as \textit{md} was introduced by Chartrand and Tian in \cite{chartrand1997distance}. It is also known as m-distance and is defined as $md(u,v)=\max\{\vec{d}(u,v),\vec{d}(v,u)\}$.

The following definitions are from \cite{chartrand1997distance}.
The $mecc(u)=\max \{md(u,v)\mid v \in V(D)\}=\max  \{ecc^+(u),ecc^-(u)\}$. The $m-radius$, $ mrad(D)$, of a digraph $D$ is defined by  $mrad(D)=\min \{mecc(v)\mid v \in V(D)\}$. The $m-diameter$, $mdiam(D)$, of a digraph $D$ is defined by  $mdiam(D)=\max \{mecc(v)\mid v \in V(D)\}$.
A digraph is $k$-self-centered if its radius and diameter are equal.
Motivated by the total eccentricity of a graph in \cite{doslic2011eccentric}, the total eccentricity of a digraph $D$ is defined as $\zeta(D)=\sum_{u\in V(D)} mecc(u)$, where $mecc(u)$ denotes the eccentricity of the vertex $u\in V(D)$.

\section{Eccentric Connectivity Index of Cartesian Product Digraphs}\label{sec:ecid}
In \cite{chakooth2025eccentric} the \emph{eccentric connectivity} index of strongly connected digraphs without loops and parallel arcs is defined as 
\begin{definition}
Let $D=(V,A)$ be a digraph, the \emph{eccentric connectivity index} of $D$ is defined as 
$$\displaystyle{\xi^C(D) = \frac{1}{2} \sum_{u \in V} ~~  (d_u^++d_u^-)\;mecc(u) }$$
where $d_u^+$, $d_u^-$ are the out-degree, in-degree, and  $mecc(u)$ is the m-eccentricity of the vertex $u$ in $V$. 
\end{definition}
The Cartesian product of digraphs is defined in \cite{hammack2018digraphs} as follows.
 \begin{definition} 
 The Cartesian product of two digraphs $D_1=(V_1,E_1)$ and $D_2=(V_2,E_2)$ with vertex sets $V_1=\{u_1,u_2,\ldots,u_m\}$ and $V_2=\{v_1,v_2,\ldots,v_n\}$ is a digraph $D=D_1\square D_2$ with vertex set $V(D)=V_1\times V_2$ in which there is an edge from vertex $(u_i,v_r)$ to the vertex $(u_j,v_s)$ if either $u_i=u_j$ and $(v_r,v_s)\in E_2$ or $v_r=v_s$ and $(u_i,u_j)\in E_1$.
 \end{definition}
 \begin{definition}
     The Cartesian product of $n$ digraphs $D_1,\ldots,D_n$ is the digraph
$D=D_1\square D_2\square\cdots\square D_n$ with vertex set
$V(D)=\{(x_1,x_2,\ldots,x_n)\mid x_i\in V(D_i)\}$ and there is an edge from the vertex
$x=(x_1,x_2,\ldots,x_n)$ to the vertex $y=(y_1,y_2,\ldots,y_n)$, when
$(x_i,y_i)\in A(D_i)$ for exactly one index $i$ $(i=1,\ldots,n)$ and
$x_j=y_j$ for each index $j\neq i$.
 \end{definition}
Let $D_1$ and $D_2$ be two strongly connected digraphs. The \emph{eccentric connectivity index} of $D_1\square D_2$ is $$\xi^C(D_1 \square D_2)
=\displaystyle{ \frac {1}{2} \sum_{(u_i,v_r)\in V(D_1 \square D_2)}}
d_{(u_i,v_r)} mecc_{D_1\square D_2}(u_i,v_r)$$ 
We first examine the m-eccentricity of vertices in the Cartesian product and establish its relationship with the corresponding m-eccentricities of the factor digraphs, for this we make use of the following definition~\ref{two-sided-ecc} and theorem~\ref{cartesian:md}, corollary~\ref{cartesian:ecc}, proposition~\ref{meccentricity equality} from \cite{tj2022study}.
 \begin{definition}[Definition 2.3.24 of  \cite{tj2022study}]\label{two-sided-ecc}
 A strongly connected digraph D is said to satisfy the two-sided eccentricity property, if for all $u_i \in V(D)$, there exist vertices $u_j ,u_k \in V(D)$ (not
necessarily distinct) such that $mecc(u_i) = \overrightarrow {d}(u_i,u_j) =\overrightarrow d(u_k,u_i).$
\end{definition}
 \begin{theorem}[Theorem 4.3.3 of \cite{tj2022study}]\label{cartesian:md}
  Let $D_1$ and $D_2$ be two strongly connected digraphs. Then
  $md_{D_1\square D_2}\big((u_i,v_r),(u_j,v_s)\big)\leq md_{D_1}(u_i,u_j)+md_{D_2}(v_r,v_s)$\\
for all $(u_i,v_r),(u_j,v_s)\in V(D_1\square D_2)$. 
 \end{theorem}
 \begin{corollary}[Theorem 4.3.4 of \cite{tj2022study}]\label{cartesian:ecc}
 Let $D_1$ and $D_2$ be two strongly connected digraphs. Then $mecc_{D_1\square D_2}(u_i,v_r)\leq mecc_{D_1}(u_i)+mecc_{D_2}(v_r)\\
$ for all $u_i\in V(D_1)$ and $v_r\in V(D_2)$.
 \end{corollary}
 \begin{proposition}[Proposition 4.3.5 of \cite{tj2022study}]\label{meccentricity equality}
Let $D_1$ and $D_2$ be two strongly connected digraphs. A necessary and sufficient condition for every vertex $(u_i,v_r)\in V(D_1\square D_2)$ to satisfy $mecc_{D_1\square D_2}(u_i,v_r)=mecc_{D_1}(u_i)+mecc_{D_2}(v_r)
$ is that either $D_1$ or $D_2$ satisfies the two-sided eccentricity property.
 \end{proposition}
   \begin{proposition}[Proposition 4.4.2 of \cite{tj2022study}]\label{Diam equality}
    Let $D_1$ and $D_2$ be two strong digraphs. Then $mdiam(D_1\square D_2)=mdiam(D_1)+mdiam(D_2).$
 \end{proposition}
We now present our main results concerning the eccentric connectivity index of Cartesian products of digraphs with respect to the metric 'maximum distance'. More precisely, the following theorem gives an explicit expression for the metric eccentric connectivity index of the Cartesian product of two strongly connected digraphs.\\
 The theorem~\ref{thm.ecc.index.cartesian} gives an upperbound for the eccentric connectivity index of cartesian product of two strongly connected digraphs. To prove the same, the following theorem~\ref{ecc-connectivity-bound} is needed and is quoted for the same of completion.
 \begin{theorem}[Theorem 4 of \cite{chakooth2025eccentric}]\label{ecc-connectivity-bound}
    If $D$ is a strongly connected digraph with $n$ vertices and $a$ arcs, then
$a.mrad(D) \leq \xi^C (D) \leq a.mdiam(D)$.  
 \end{theorem}
 Applying the theorem ~\ref{ecc-connectivity-bound}, it can be seen that $\xi^C(D_1 \square D_2)$ is bounded by $\bigl(n_1e_2+n_2e_1\bigl)\bigl(mdiam(D_1)+ mdiam(D_2)\bigl)$, which may be quite large. But the theorem \ref{thm.ecc.index.cartesian} will give an upperbound which is very much less than $\bigl(n_1e_2+n_2e_1\bigl)\bigl(mdiam(D_1)+ mdiam(D_2)\bigl)$ for general strongly connected digraphs.
 \begin{theorem}\label{thm.ecc.index.cartesian}
     If $D_1$ and $D_2$ are two strongly connected digraphs with $n_1,n_2$ vertices and $e_1,e_2$ arcs respectively, then $\xi^C(D_1 \square D
     _2) \le n_2\xi^C(D_1)+n_1\xi^C(D_2)+e_1\zeta(D_2)+e_2\zeta(D_1)$.  
 \end{theorem}
 \begin{proof}
 Using the definition of eccentric connectivity index of digraphs, 
     \begin{align*}
\xi^C(D_1 \square D_2)
&=\frac {1}{2} \sum_{(u_i,v_r)\in V(D_1 \square D_2)}
d_{(u_i,v_r)} mecc_{D_1\square D_2}(u_i,v_r)\\ & \le \frac {1}{2} \sum_{(u_i,v_r)\in V(D_1 \square D_2)}
d_{(u_i,v_r)}\left[ mecc_{D_1}(u_i)+mecc_{D_2}(v_r) \right ]\\
&=\frac {1}{2}\sum_{v_r\in V(D_2)}\sum_{u_i\in V(D_1)}
d_{u_i}mecc_{D_1}(u_i)
+\frac {1}{2}\sum_{u_i\in V(D_1)}\sum_{v_r\in V(D_2)}
d_{v_r}mecc_{D_2}(v_r)\\
&\quad+\frac {1}{2}\sum_{u_i\in V(D_1)} d_{u_i}
\sum_{v_r\in V(D_2)}mecc_{D_2}(v_r)
+\frac {1}{2}\sum_{v_r\in V(D_2)}d_{v_r}
\sum_{u_i\in V(D_1)}mecc_{D_1}(u_1)\\
&= n_2\xi^C(D_1)+n_1 \xi^C(D_2)
+e_1\zeta(D_2)+e_2 \zeta(D_1).
\end{align*}
Hence $\xi^C(D_1 \square D
     _2) \le n_2\xi^C(D_1)+n_1\xi^C(D_2)+e_1\zeta(D_2)+e_2\zeta(D_1)$.\qed
 \end{proof} 
  By Proposition \ref{meccentricity equality}, $mecc_{D_1\square D_2}(u_i,v_r)=mecc_{D_1}(u_i)+mecc_{D_2}(v_r)$ if and only if  either $D_1$ or $D_2$ satisfies the two-sided eccentricity property. Then we have the following Corollary.
\begin{corollary}\label{equality}
    Equality of the above theorem holds if and only if $D_1$ or $D_2$ satisfy the two-sided eccentricity property.
\end{corollary}
\begin{corollary}
    If $D$ is a strongly connected digraph which satisfies the two-sided eccentricity property then
    $\xi^C(D \square D
     ) = 2n\xi^C(D)+2e\zeta(D)$.
\end{corollary}
 The equality of the theorem \ref{thm.ecc.index.cartesian} holds if and only if $D_1 \square D_2$ is self-centered.   The necessary and sufficient condition for $D_1 \square D_2$ to be self-centered is given in the following theorem. 
\begin{theorem}
Let $D_1$ and $D_2$ be two self-centered strongly connected digraphs. If  $D_1$ or $D_2$ satisfy the two-sided eccentricity property, then the Cartesian product $D_1\square D_2$ is self-centered . 
\end{theorem}
\begin{proof}
    Let $D_1$ and $D_2$ be self-centered digraphs with  $mrad(D_i)=k_i=mdiam(D_i)$ for $i=1,2$. 
    Without loss of generality assume that $D_1$ satisfy the two-sided eccentricity property. \\Then by proposition \ref{meccentricity equality}, $mecc_{D_1\square D_2}(u_i,v_r)=mecc_{D_1}(u_i)+mecc_{D_2}(v_r)=k_1+k_2$. 
    Thus $mecc_{D_1\square D_2}(u_i,v_r)=k$, where $k=k_1+k_2$.
    Hence $D_1\square D_2$ is self-centered.
    \qed
\end{proof} 
In general it can be seen that the cartesian product of two self-centered strongly connected may not be self-centered even if $D_1$ and $D_2$ are self-centered, See example \ref{self.centered.eg1}.
\begin{example}\label{self.centered.eg1}
    Cartesian product of two self-centerd strongly connected digraphs need not be self-centered.  
    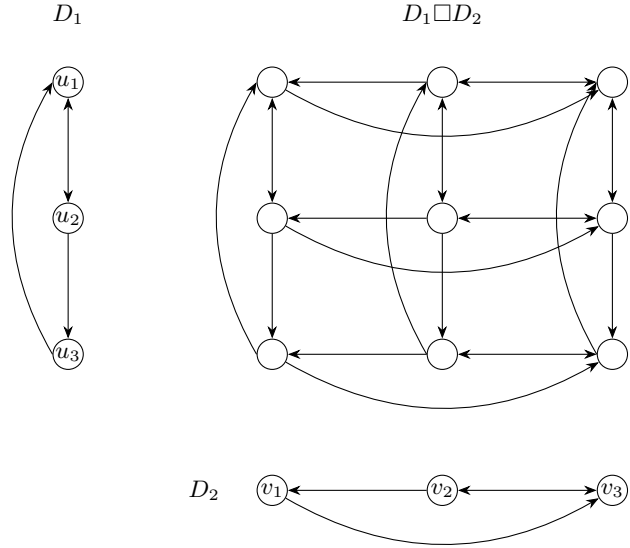
\begin{figure}[htbp]
\centering
\begin{tikzpicture}[
>=Stealth,
vertex/.style={circle,draw,minimum size=4mm,inner sep=0pt},
scale=.9
]

\node at (-2,3) {$D_1$};

\node[vertex] (u1) at (-2,2) {$u_1$};
\node[vertex] (u2) at (-2,0) {$u_2$};
\node[vertex] (u3) at (-2,-2) {$u_3$};

\draw[<->] (u1)--(u2);
\draw[->] (u2)--(u3);
\draw[->] (u3.west) to[out=120,in=240] (u1.west);

\node at (0,-4) {$D_2$};

\node[vertex] (v1) at (1,-4) {$v_1$};
\node[vertex] (v2) at (3.5,-4) {$v_2$};
\node[vertex] (v3) at (6,-4) {$v_3$};

\draw[<-] (v1)--(v2);
\draw[<->] (v2)--(v3);
\draw[<-,bend left=30] (v3) to (v1);

\node[draw=none] at (3.5,3) {$D_1\square D_2$};

\foreach \i/\x in {1/1,2/3.5,3/6}
{
\node[vertex] (p\i1) at (\x,2) {};
\node[vertex] (p\i2) at (\x,0) {};
\node[vertex] (p\i3) at (\x,-2) {};

\draw[<->] (p\i1)--(p\i2);
\draw[->] (p\i2)--(p\i3);
\draw[->] (p\i3.west) to[out=120,in=240] (p\i1.west);
}

\foreach \j in {1,2,3}
{
\draw[<-] (p1\j)--(p2\j);
\draw[<->] (p2\j)--(p3\j);
\draw[<-,bend left=30] (p3\j) to (p1\j);
}

\end{tikzpicture}
\caption{The digraph $D_1\square D_2$ is not self-centered even though  $D_1$ and $D_2$ are self-centered.}
\label{fig:D1D2}
\end{figure}
\end{example}

\begin{theorem}
    Let $D_1$ and $D_2$ be two strong digraphs with with $n_1,n_2$ vertices and $e_1, e_2$ arcs respectively such that $D_1 \square D_2$ is $k-$self-centered, Then $\xi^C(D_1 \square D_2)= k\big(n_1e_2+e_1n_2\big)$. 
\end{theorem}
\begin{proof}
    Since $D_1 \square  D_2$ is self-centered digraph with $n_1e_2+e_1n_2$ arcs, $mrad(D_1 \square D_2)=mdiam(D_1 \square \ D_2)=k$.  
    Then by corollory of Theorem \ref{ecc-connectivity-bound} of \cite{chakooth2025eccentric}, We have 
    $$\xi^C(D_1 \square D_2) =(n_1e_2+e_1n_2)madiam(D_1 \square D_2) =k(n_1e_2+e_1n_2).$$\qed
\end{proof}

Extending the above results for $n$ digraphs we have the following.

\begin{theorem}
Let $D_1,D_2,\ldots,D_n$ be strongly connected digraphs with
$n_i=|V(D_i)|$ and $e_i=|A(D_i)|$, for $1\leq i\leq n$. Then
$$
\xi^C(D_1\square D_2\square\cdots\square D_n)
\leq
\sum_{i=1}^n
\left(\prod_{j\ne i}n_j\right)\xi^C(D_i)
+
\sum_{i=1}^n
\left(
\sum_{k\ne i}e_k
\prod_{\ell\ne i,k}n_\ell
\right)\zeta(D_i).
$$

\end{theorem}

\begin{proof}
Let $D=D_1\square D_2\square\cdots\square D_n$. By definition,

$$
\xi^C(D)
=
\frac{1}{2}
\sum_{(v_1,\ldots,v_n)\in V(D)}
d_D(v_1,\ldots,v_n)\,
mecc_D(v_1,\ldots,v_n).
$$

Since
$mecc_D(v_1,\ldots,v_n)
\leq
\sum_{i=1}^n mecc_{D_i}(v_i)
$
and
$
d_D(v_1,\ldots,v_n)
=
\sum_{k=1}^n d_{D_k}(v_k),
$

we have

$$
\xi^C(D)
\leq
\frac{1}{2}
\sum_{(v_1,\ldots,v_n)\in V(D)}
\left(\sum_{k=1}^n d_{D_k}(v_k)\right)
\left(\sum_{i=1}^n mecc_{D_i}(v_i)\right).
$$

Separating the terms corresponding to $k=i$ and $k\ne i$, we obtain

$$
\begin{aligned}
\xi^C(D)\leq {}&
\frac{1}{2}\sum_{i=1}^n
\sum_{(v_1,\ldots,v_n)\in V(D)}
d_{D_i}(v_i)\,mecc_{D_i}(v_i)\\
&+
\frac{1}{2}\sum_{i=1}^n\sum_{k\ne i}
\sum_{(v_1,\ldots,v_n)\in V(D)}
d_{D_k}(v_k)\,mecc_{D_i}(v_i).
\end{aligned}
$$

Now,

$$
\frac{1}{2}
\sum_{(v_1,\ldots,v_n)\in V(D)}
d_{D_i}(v_i)\,mecc_{D_i}(v_i)
=
\left(\prod_{j\ne i}n_j\right)\xi^C(D_i),
$$

while, for $k\ne i$,

$$
\begin{aligned}
&\frac{1}{2}
\sum_{(v_1,\ldots,v_n)\in V(D)}
d_{D_k}(v_k)\,mecc_{D_i}(v_i)\\
&\quad =
\frac{1}{2}
\left(\prod_{\ell\ne i,k}n_\ell\right)
\left(\sum_{v_k\in V(D_k)}d_{D_k}(v_k)\right)
\left(\sum_{v_i\in V(D_i)}mecc_{D_i}(v_i)\right)\\
&\quad =
e_k\left(\prod_{\ell\ne i,k}n_\ell\right)\zeta(D_i),
\end{aligned}
$$

where we have used

$
\sum_{v_k\in V(D_k)}d_{D_k}(v_k)=2e_k
$ and $
\zeta(D_i)=\sum_{v_i\in V(D_i)}mecc_{D_i}(v_i).
$\\

Therefore,

$$
\xi^C(D)
\leq
\sum_{i=1}^n
\left(\prod_{j\ne i}n_j\right)\xi^C(D_i)
+
\sum_{i=1}^n
\left(
\sum_{k\ne i}e_k
\prod_{\ell\ne i,k}n_\ell
\right)\zeta(D_i).
$$
\qed
\end{proof}

\section{Eccentric Connectivity Index of Strong Product Digraphs}
In this section, the  eccentric connectivity index of the strong product of strongly connected digraphs is studied. First recall some definitions 
from \cite{hammack2011handbook}.
\begin{definition}
The strong product $D_1 \boxtimes D_2$ of two digraphs $D_1$ and $D_2$ with vertex sets $V(D_1) = \{u_1,u_2,\ldots,u_m\}$ and $V(D_2) = \{v_1,v_2,\ldots,v_n\}$ is the digraph having the vertex set $V(D_1) \times V(D_2)$ and arc set $A(D_1 \boxtimes D_2)$ defined as follows.

There is an edge from the vertex $(u_i,v_r)$ to $(u_j,v_s)$ in $D_1 \boxtimes D_2$ if either

1. $(u_i,u_j) \in A(D_1)$, $v_r = v_s$, or

2. $u_i = u_j$, $(v_r,v_s) \in A(D_2)$, or

3. $(u_i,u_j) \in A(D_1)$, $(v_r,v_s) \in A(D_2)$.\\
\end{definition}
\begin{definition}
The strong product of $n$ digraphs $D_1,\ldots,D_n$ is the digraph $D = D_1 \boxtimes D_2 \boxtimes \cdots \boxtimes D_n$ with vertex set $V(D) = \{(x_1,x_2,\ldots,x_n) \mid x_i \in V(D_i)\}$ and there is an edge from a vertex $(x_1,x_2,\ldots,x_n)$ to the vertex $(y_1,y_2,\ldots,y_n)$ provided 
$x_i = y_i$ or $(x_i,y_i) \in A(D_i)$ for all $i \in [n]$ and $(x_i,y_i) \in A(D_i)$ for at least one $i \in [n]$.
\end{definition}

The distance between two vertices $(g,h)$ and $(g',h')$ in the strong product $G \boxtimes H$ of two graphs $G$ and $H$ is
$d_{G \boxtimes H}((g,h),(g',h')) = \max\{d_G(g,g'),d_H(h,h')\}$.
From \cite{tj2022study} we have the following.

\begin{lemma}[Lemma 5.2.1 of \cite{tj2022study}]
Let $D_1$ and $D_2$ be two strongly connected digraphs. Then
$$md_{D_1 \boxtimes D_2} ((u_i,v_r),(u_j ,v_s)) = \max \{ md_{D_1} (u_i,u_j ),md_{D_2} (v_r,v_s)\}$$
$$mecc_{D_1 \boxtimes D_2} (u_i,v_r) = \max \{ mecc_{D_1} (u_i),mecc_{D_2} (v_r)\}.$$
\end{lemma}
\begin{corollary}[Corollary 5.2.2 of \cite{tj2022study}]
    Let $D_1$ and $D_2$ be two strongly connected digraphs. Then
$$mrad(D_1 \boxtimes D_2) = \max \{mrad(D_1),mrad(D_2)\},$$
$$mdiam(D_1 \boxtimes D_2) = \max \{mdiam(D_1),mdiam(D_2)\}.$$
\end{corollary}
We now study the metric eccentric connectivity index, defined with respect to the metric maximum distance, for strong products of strongly connected digraphs.
The following lemma, together with Theorem~\ref{ecc-connectivity-bound}, immediately yields Theorem~\ref{thm.eci.strong.bdd1}.

\begin{lemma}\label{strong.self-centered}
Let $D_1$ and $D_2$ be strongly connected digraphs. Then the strong product $D_1\boxtimes D_2$ is self-centered if and only if, for some $i\in\{1,2\}$, $D_i$ is self-centered and
$mdiam(D_j)\leq mrad(D_i)$
where $j\in\{1,2\}$ and $j\neq i$.
\end{lemma}
\begin{proof}

Suppose, for some $i\in \{1,2\}$, $D_i$ is self-centered and $mdiam(D_j)\leq mrad(D_i)$, where $j\neq i$. Then $mecc_{D_i}(u)=mrad(D_i)$ for every $u\in V(D_i)$ and $mecc_{D_j}(v)\leq mdiam(D_j)\leq mrad(D_i)$ for every $v\in V(D_j)$. \\
Hence $mecc_{D_1\boxtimes D_2}(u,v)=mrad(D_i)$ for every $(u,v)\in V(D_1\boxtimes D_2)$, and therefore $D_1\boxtimes D_2$ is self-centered.

Conversely, suppose $D_1\boxtimes D_2$ is self-centered with $mecc(u_i,v_r)=k$. Then $\max\{mecc_{D_1}(u),mecc_{D_2}(v)\}=k$, for all $u\in V(D_1)$ and $v\in V(D_2)$. Thus, at least one factor, say $D_i$, has $mecc_{D_i}(u)=K$ for every $u\in V(D_i)$, and hence $D_i$ is self-centered. Also, $mecc_{D_j}(v)\leq k=mrad(D_i)$ for every $v\in V(D_j)$, and consequently $mdiam(D_j)\leq mrad(D_i)$.\qed
\end{proof}

\begin{example}
Strong product of two strongly connected digraphs need not be self-centered even if one of them is self-centered.
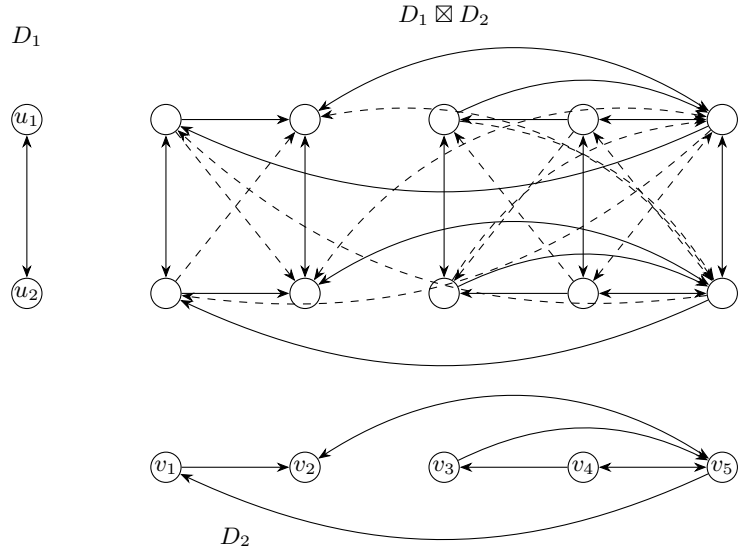
\begin{figure}[h]
\centering
\begin{tikzpicture}[
>=Stealth,
vertex/.style={circle,draw,minimum size=4mm,inner sep=0pt},
scale=0.92
]

\node at (-2,3.2) {$D_1$};
\node[vertex] (u1) at (-2,2) {$u_1$};
\node[vertex] (u2) at (-2,-0.5) {$u_2$};

\draw[<->] (u1)--(u2);

\node at (1,-4) {$D_2$};

\node[vertex] (v1) at (0,-3) {$v_1$};
\node[vertex] (v2) at (2,-3) {$v_2$};
\node[vertex] (v3) at (4,-3) {$v_3$};
\node[vertex] (v4) at (6,-3) {$v_4$};
\node[vertex] (v5) at (8,-3) {$v_5$};

\draw[->]  (v1)--(v2);
\draw[<->,bend left=33] (v2) to (v5);
\draw[->,bend left=26] (v3) to (v5);
\draw[->]  (v4)--(v3);
\draw[<->] (v4)--(v5);
\draw[->,bend left=25] (v5) to (v1);

\node at (4,3.5) {$D_1\boxtimes D_2$};


\node[vertex] (p11) at (0,2) {};
\node[vertex] (p12) at (2,2) {};
\node[vertex] (p13) at (4,2) {};
\node[vertex] (p14) at (6,2) {};
\node[vertex] (p15) at (8,2) {};

\node[vertex] (p21) at (0,-0.5) {};
\node[vertex] (p22) at (2,-0.5) {};
\node[vertex] (p23) at (4,-0.5) {};
\node[vertex] (p24) at (6,-0.5) {};
\node[vertex] (p25) at (8,-0.5) {};

\draw[<->] (p11)--(p21);
\draw[<->] (p12)--(p22);
\draw[<->] (p13)--(p23);
\draw[<->] (p14)--(p24);
\draw[<->] (p15)--(p25);

\draw[->] (p11)--(p12);

\draw[<->,bend left=33] (p12) to (p15);

\draw[->,bend left=26] (p13) to (p15);

\draw[->] (p14)--(p13);

\draw[<->] (p14)--(p15);

\draw[->,bend left=25] (p15) to (p11);

\draw[->] (p21)--(p22);

\draw[<->,bend left=33] (p22) to (p25);

\draw[->,bend left=26] (p23) to (p25);

\draw[->] (p24)--(p23);

\draw[<->] (p24)--(p25);

\draw[->,bend left=25] (p25) to (p21);

\draw[->,dashed] (p11) to (p22);
\draw[->,dashed] (p21) to (p12);

\draw[<->,dashed,bend left=33] (p12) to (p25);
\draw[<->,dashed,bend left=33] (p22) to (p15);

\draw[->,dashed,bend left=26] (p13) to (p25);
\draw[->,dashed,bend left=26] (p23) to (p15);

\draw[->,dashed] (p14) to (p23);
\draw[->,dashed] (p24) to (p13);

\draw[<->,dashed] (p14) to (p25);
\draw[<->,dashed] (p24) to (p15);

\draw[->,dashed,bend left=25] (p15) to (p21);
\draw[->,dashed,bend left=25] (p25) to (p11);

\end{tikzpicture}

\caption{The digraphs $D_1$ is self-centered but the strong product
$D_1\boxtimes D_2$ is  not self-centered.}

\end{figure}
\end{example}
As in the case of Cartesian product, Applying the theorem ~\ref{ecc-connectivity-bound}, immediately yeilds Theorem \ref{thm.eci.strong.bdd1}.
\begin{theorem}\label{thm.eci.strong.bdd1}
     Let $D_1$ and $D_2$ be two strongly connected digraphs with $e_1,e_2$ arcs respectively. Then $(n_2e_1+n_1e_2+e_1e_2) \max\{mrad(D_1), mrad(D_2)\}\le\xi^C(D_1 \boxtimes D_2) \le (n_2e_1+n_1e_2+e_1e_2) \max\{mdiam(D_1), mdiam(D_2)\}$.
 \end{theorem}
 Equality of the above theorem holds if and only if $D_i$, $i\in{1,2}$, is self-centered and $
mdiam(D_j)\leq mrad(D_i)$
where $j\in{1,2}$ and $j\neq i$ is obvious from theorem \ref{ecc-connectivity-bound} and from lemma \ref{strong.self-centered}.
 
\begin{theorem}\label{thm.eci.strong.bdd2}
Let $D_1$ and $D_2$ be two strongly connected digraphs with and $e_1$ and $e_2$ arcs respectively. Then\\
$\begin{aligned}
\xi^C(D_1\boxtimes D_2)
&<\,(n_2+e_2)\xi^C(D_1)
+(n_1+e_1)\xi^C(D_2)+e_2\zeta(D_1)
+e_1\zeta(D_2)\\
&-\frac{1}{2}\sum_{i=1}^{n_1}\sum_{r=1}^{n_2}
d_D(u_i,v_r)\min\{\alpha_i,\beta_r\}.
\end{aligned}
$

\end{theorem}
\begin{proof}
Let $D=D_1\boxtimes D_2$, where $
V(D_1)=\{u_1,u_2,\ldots,u_{n_1}\}
\quad\text{and}\quad
V(D_2)=\{v_1,v_2,\ldots,v_{n_2}\}.
$
Put
$\alpha_i=mecc_{D_1}(u_i)
\quad\text{and}\quad
\beta_r=mecc_{D_2}(v_r).
$ So, $\alpha_i \ge1$ and $\beta_r\ge1$.  Since
$
mecc_D(u_i,v_r)=\max\{\alpha_i,\beta_r\},
$
and $
\max\{\alpha_i,\beta_r\}
=
\alpha_i+\beta_r-\min\{\alpha_i,\beta_r\},
$
we have
$$
\begin{aligned}
\xi^C(D_1\boxtimes D_2)
&=\frac{1}{2}\sum_{i=1}^{n_1}\sum_{r=1}^{n_2}
d_D(u_i,v_r)(\alpha_i+\beta_r)\\
&-\frac{1}{2}\sum_{i=1}^{n_1}\sum_{r=1}^{n_2}
d_D(u_i,v_r)\min\{\alpha_i,\beta_r\}.
\end{aligned}
$$
The degree $d_D(u_i,v_r)=
d_{D_1}(u_i)+d_{D_2}(v_r)
+d_{D_1}^{+}(u_i)d_{D_2}^{+}(v_r)
+d_{D_1}^{-}(u_i)d_{D_2}^{-}(v_r)$ of the vertex $(u_i,v_r)$ in $D_1\boxtimes D_2$ can be deduced from  \cite{bovzovic2020efficient}. 

Now,

$$
\begin{aligned}
\frac{1}{2}\sum_{i=1}^{n_1}\sum_{r=1}^{n_2}
d_D(u_i,v_r)\alpha_i
&=
\frac{1}{2}\sum_{i=1}^{n_1}\sum_{r=1}^{n_2}
d_{D_1}(u_i)\alpha_i\\
&+\frac{1}{2}\sum_{i=1}^{n_1}\sum_{r=1}^{n_2}
d_{D_2}(v_r)\alpha_i\\
&+\frac{1}{2}\sum_{i=1}^{n_1}\sum_{r=1}^{n_2}
d_{D_1}^{+}(u_i)d_{D_2}^{+}(v_r)\alpha_i\\
&+\frac{1}{2}\sum_{i=1}^{n_1}\sum_{r=1}^{n_2}
d_{D_1}^{-}(u_i)d_{D_2}^{-}(v_r)\alpha_i.
\end{aligned}
$$

Using
$\sum_{r=1}^{n_2}d_{D_2}(v_r)=2e_2,
\qquad
\sum_{r=1}^{n_2}d_{D_2}^{+}(v_r)
=
\sum_{r=1}^{n_2}d_{D_2}^{-}(v_r)=e_2,$

we obtain

$$
\begin{aligned}
\frac{1}{2}\sum_{i=1}^{n_1}\sum_{r=1}^{n_2}
d_D(u_i,v_r)\alpha_i
&=
\frac{n_2}{2}\sum_{i=1}^{n_1}d_{D_1}(u_i)\alpha_i
+e_2\sum_{i=1}^{n_1}\alpha_i\\
&+
\frac{e_2}{2}\sum_{i=1}^{n_1}
\bigl[d_{D_1}^{+}(u_i)+d_{D_1}^{-}(u_i)\bigr]\alpha_i\\
&=
(n_2+e_2)\xi^C(D_1)+e_2\zeta(D_1).
\end{aligned}
$$

Similarly,

$$
\frac{1}{2}\sum_{i=1}^{n_1}\sum_{r=1}^{n_2}
d_D(u_i,v_r)\beta_r
=
(n_1+e_1)\xi^C(D_2)+e_1\zeta(D_2).
$$

Consequently,
$$
\begin{aligned}
\xi^C(D_1\boxtimes D_2)
&=(n_2+e_2)\xi^C(D_1)
+(n_1+e_1)\xi^C(D_2)+e_2\zeta(D_1)+e_1\zeta(D_2)\\
&-\frac{1}{2}\sum_{i=1}^{n_1}\sum_{r=1}^{n_2}
d_D(u_i,v_r)\min\{\alpha_i,\beta_r\}.
\end{aligned}
$$\qed
\end{proof}

\begin{theorem}
     If $D_1$ and $D_2$ are two strongly connected digraphs with $e_1,e_2$ arcs respectively such that $mecc(u_i)\le mecc(v_r)$ for $i=1,2,\cdots,n_1,r=1,2,\cdots,n_2, $ then $$\xi^C(D_1\boxtimes D_2)
=\big(n_1+e_1\big) \xi^C(D_2)
+e_1\zeta(D_2).$$
 \end{theorem}
\begin{proof}
Put $
\alpha_i=mecc_{D_1}(u_i),\quad
\beta_r=mecc_{D_2}(v_r).
$  From the proof of Theorem~\ref{thm.eci.strong.bdd2}, we have

$$
\begin{aligned}
\xi^C(D_1\boxtimes D_2)
={}&(n_2+e_2)\xi^C(D_1)
+(n_1+e_1)\xi^C(D_2)\\
&+e_2\zeta(D_1)+e_1\zeta(D_2)\\
&-\frac{1}{2}\sum_{i=1}^{n_1}\sum_{r=1}^{n_2}
d_D(u_i,v_r)\min\{\alpha_i,\beta_r\}.
\end{aligned}
$$

If $\alpha_i\leq\beta_r$ for all $i,r$, then
$\min{\alpha_i,\beta_r}=\alpha_i$. Also, from the same proof,

$$
\frac{1}{2}\sum_{i=1}^{n_1}\sum_{r=1}^{n_2}
d_D(u_i,v_r)\alpha_i
=(n_2+e_2)\xi^C(D_1)+e_2\zeta(D_1).
$$

Hence $
\xi^C(D_1\boxtimes D_2)
=(n_1+e_1)\xi^C(D_2)+e_1\zeta(D_2).
$  Similarly, if $\alpha_i\geq\beta_r$ for all $i,r$, then
$\min{\alpha_i,\beta_r}=\beta_r$, and $
\frac{1}{2}\sum_{i=1}^{n_1}\sum_{r=1}^{n_2}
d_D(u_i,v_r)\beta_r
=(n_1+e_1)\xi^C(D_2)+e_1\zeta(D_2).
$
Therefore, $
\xi^C(D_1\boxtimes D_2)
=(n_2+e_2)\xi^C(D_1)+e_2\zeta(D_1).
$\qed
\end{proof}
 \begin{corollary}
Let $D_1$ and $D_2$ be two strongly connected digraphs with $e_1$ and $e_2$ arcs respectively. If $mecc(u_i)\ge mecc(v_r)$
for $i=1,2,\ldots,n_1$ and $r=1,2,\ldots,n_2$, then
$
\xi^C(D_1\boxtimes D_2)
=
\bigl(n_2+e\bigr)\xi^C(D_1)
+
e_1\zeta(D_1).
$
\end{corollary}
Expanding the eccentric connectivity index for n digraphs we have the following.
\begin{theorem}
Let $D_1,D_2,\ldots,D_n$ be strongly connected digraphs. Then\\
$\begin{aligned}
\xi^C(D_1\boxtimes\cdots\boxtimes D_n)&=
\frac{1}{2}
\sum_{v_i\in V(D_i)}
\left[
\prod_{i=1}^{n}\left(1+d^+_{D_i}(v_i)\right)
+
\prod_{i=1}^{n}\left(1+d^-_{D_i}(v_i)\right)-2
\right]\\
&\times 
\max_{1\le i\le n}\left\{mecc_{D_i}(v_i)\right\}.\end{aligned}$
\end{theorem}
 
\section{Conclusion}

In this paper, the eccentric connectivity index with respect to the metric maximum distance is studied for Cartesian and strong products of strongly connected digraphs. Bounds for the eccentric connectivity index of both products and the corresponding equality cases are obtained, and the results are extended to products of several digraphs. For the Cartesian product $D_1\square D_2$, the upper bound obtained in terms of the parameters of the factor digraphs improves the corresponding diameter-based upper bound. The converse part of the  Proposition \ref{meccentricity equality} given in \cite{tj2022study} need not be true when $D_1=D_2$, see figure \ref{fig:D-cartesian-D}.
\begin{figure}[H]
\centering

\begin{tikzpicture}[
    >=Stealth,
    vertex/.style={
        circle,
        draw,
        minimum size=4mm,
        inner sep=1pt,
        font=\scriptsize, 
    },
    ecc/.style={
        draw=none,
        font=\scriptsize
    },
    scale=.7
]

\node [vertex] (a) at (-2.5,0) {$a$};

\node [vertex] (b) at (-2.5,-2.5) {$b$};

\node [vertex] (c) at (-2.5,-5) {$c$};

\node [vertex] (d) at (-2.5,-7.5) {$d$};

\draw[->,bend right=14] (a) to (b);
\draw[->,bend right=14] (b) to (a);

\draw[->,bend right=22] (a) to (c);
\draw[->,bend right=22] (c) to (a);

\draw[->,bend right=18] (b) to (d);

\draw[->] (c) to (b);

\draw[->,bend right=25] (d) to (a);

\node [vertex] (a) at (0,-10) {$a$};

\node [vertex] (b) at (2.5,-10) {$b$};

\node [vertex] (c) at (5,-10) {$c$};

\node [vertex] (d) at (7.5,-10) {$d$};

\draw[->,bend left=14] (a) to (b);
\draw[->,bend left=14] (b) to (a);

\draw[->,bend left=22] (a) to (c);
\draw[->,bend left=22] (c) to (a);

\draw[->,bend left=18] (b) to (d);

\draw[->] (c) to (b);

\draw[->,bend left=25] (d) to (a);

\node[vertex] (aa) at (0,0) {};

\node[vertex] (ba) at (2.5,0) {};

\node[vertex] (ca) at (5,0) {};

\node[vertex] (da) at (7.5,0) {};

\node[vertex] (ab) at (0,-2.5) {};

\node[vertex] (bb) at (2.5,-2.5) {};

\node[vertex] (cb) at (5,-2.5) {};

\node[vertex] (db) at (7.5,-2.5) {};

\node[vertex] (ac) at (0,-5) {};

\node[vertex] (bc) at (2.5,-5) {};

\node[vertex] (cc) at (5,-5) {};

\node[vertex] (dc) at (7.5,-5) {};

\node[vertex] (ad) at (0,-7.5) {};

\node[vertex] (bd) at (2.5,-7.5) {};

\node[vertex] (cd) at (5,-7.5) {};

\node[vertex] (dd) at (7.5,-7.5) {};

%

\draw[->,bend left=14] (aa) to (ba);
\draw[->,bend left=14] (ba) to (aa);

\draw[->,bend left=22] (aa) to (ca);
\draw[->,bend left=22] (ca) to (aa);

\draw[->,bend left=18] (ba) to (da);

\draw[->] (ca) to (ba);

\draw[->,bend left=30] (da) to (aa);

\draw[->,bend left=14] (ab) to (bb);
\draw[->,bend left=14] (bb) to (ab);

\draw[->,bend left=22] (ab) to (cb);
\draw[->,bend left=22] (cb) to (ab);

\draw[->,bend left=18] (bb) to (db);

\draw[->] (cb) to (bb);

\draw[->,bend left=30] (db) to (ab);

\draw[->,bend left=14] (ac) to (bc);
\draw[->,bend left=14] (bc) to (ac);

\draw[->,bend left=22] (ac) to (cc);
\draw[->,bend left=22] (cc) to (ac);

\draw[->,bend left=18] (bc) to (dc);

\draw[->] (cc) to (bc);

\draw[->,bend left=30] (dc) to (ac);

\draw[->,bend left=14] (ad) to (bd);
\draw[->,bend left=14] (bd) to (ad);

\draw[->,bend left=22] (ad) to (cd);
\draw[->,bend left=22] (cd) to (ad);

\draw[->,bend left=18] (bd) to (dd);

\draw[->] (cd) to (bd);

\draw[->,bend left=30] (dd) to (ad);


\draw[->,bend right=14] (aa) to (ab);
\draw[->,bend right=14] (ab) to (aa);

\draw[->,bend right=22] (aa) to (ac);
\draw[->,bend right=22] (ac) to (aa);

\draw[->,bend right=18] (ab) to (ad);

\draw[->] (ac) to (ab);

\draw[->,bend right=30] (ad) to (aa);

\draw[->,bend right=14] (ba) to (bb);
\draw[->,bend right=14] (bb) to (ba);

\draw[->,bend right=22] (ba) to (bc);
\draw[->,bend right=22] (bc) to (ba);

\draw[->,bend right=18] (bb) to (bd);

\draw[->] (bc) to (bb);

\draw[->,bend right=30] (bd) to (ba);

\draw[->,bend right=14] (ca) to (cb);
\draw[->,bend right=14] (cb) to (ca);

\draw[->,bend right=22] (ca) to (cc);
\draw[->,bend right=22] (cc) to (ca);

\draw[->,bend right=18] (cb) to (cd);

\draw[->] (cc) to (cb);

\draw[->,bend right=30] (cd) to (ca);

\draw[->,bend right=14] (da) to (db);
\draw[->,bend right=14] (db) to (da);

\draw[->,bend right=22] (da) to (dc);
\draw[->,bend right=22] (dc) to (da);

\draw[->,bend right=18] (db) to (dd);

\draw[->] (dc) to (db);

\draw[->,bend right=30] (dd) to (da);

\end{tikzpicture}

\caption{The digraph $D$ and Cartesian product $D\square D$ are self-centered, but $D$ does not satisfy two-sided eccentricity property; the $m$-eccentricity is displayed in red.}
\label{fig:D-cartesian-D}
\end{figure}
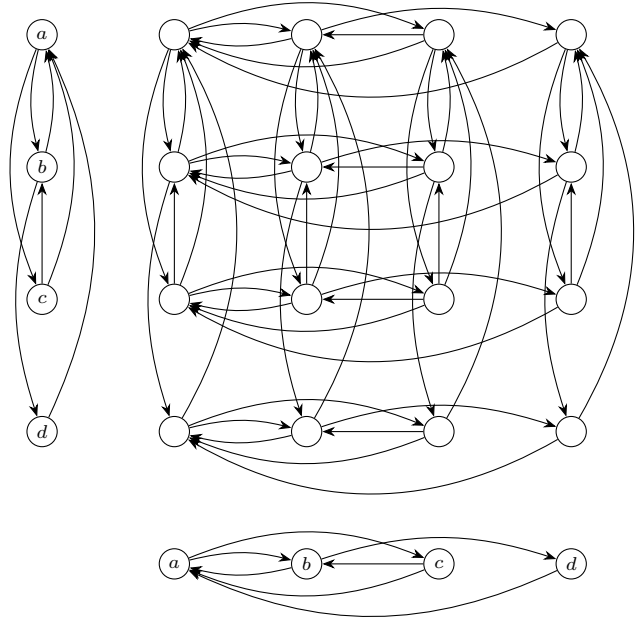

\bibliographystyle{splncs04}
\bibliography{ref}

\end{document}